\documentclass[b4]{article}
\usepackage[top=30truemm,bottom=30truemm,left=27truemm,right=27truemm]{geometry}
\usepackage{lipsum} 
\usepackage{titlesec}
\titleformat{\subsection}[runin]
       {\normalfont\bfseries}
       {\thesubsection}
       {0.5em}
       {}
       [.]
\usepackage{amsmath,amsthm,amssymb,mathdots}
\usepackage{setspace,shuffle}
\usepackage{color}
\allowdisplaybreaks

\newtheorem{theorem}{Theorem}[section]
\newtheorem{proposition}[theorem]{Proposition}

\newtheorem{conjecture}[theorem]{Conjecture}

\newtheorem{problem}[theorem]{Problem}

\theoremstyle{remark}
\newtheorem{remark}[theorem]{Remark}
\newtheorem{example}[theorem]{Example}

\numberwithin{equation}{section}

\newcommand{\R}{\mathbb R}
\newcommand{\Z}{{\mathbb Z}}
\newcommand{\N}{{\mathbb N}}
\newcommand{\C}{{\mathbb C}}
\newcommand{\Q}{{\mathbb Q}}

\begin{document}
\begin{center} 
{\Large Catalogue of modular relations for double zeta values}  \\ \vspace{2mm}
Koji Tasaka   \\ 

\vspace{2mm}

School of Information Science and Technology \\ Aichi Prefectural University 
\end{center}


\section{Introduction}

\subsection{Background}
At the conference, based on my works, I overviewed the studies of modular relations of multiple zeta values
\[ \zeta(k_1,\ldots,k_d):=\sum_{0<m_1<\cdots<m_d}\frac{1}{m_1^{k_1}\cdots m_d^{k_d}}\in \R \quad (k_1,\ldots,k_d\in \N, \ k_d\ge2).\]
As usual, a tuple $\boldsymbol{k}=(k_1,\ldots,k_d)$ of positive integers is called an index and its weight and depth are defined by ${\rm wt}(\boldsymbol{k})=k_1+\cdots+k_d$ and ${\rm dep}(\boldsymbol{k})=d$, respectively.
Roughly speaking, modular relations are (homogeneous) $\Q$-linear relations $ \sum_{{\rm wt}(\boldsymbol{k})=k} a_{\boldsymbol{k}}\, \zeta(\boldsymbol{k}) =0$ of multiple zeta values of fixed weight $k$ whose coefficients $a_{\boldsymbol{k}}\in \Q$ are `originated' from modular forms on (a subgroup of) the full modular group $\Gamma_1:={\rm SL}_2(\Z)$.
Several kinds of such relations, in particular, for double zeta values have been found up to now.
It is our purpose to review these results and also succinctly outline potential avenues for future research projects.

The existence of modular relations for double zeta values was initially noted by Zagier \cite{Zagier94,Zagier93}.
It was observed that for each positive even integer $k$, the values $\zeta(1,k-1),\zeta(3,k-3),\ldots,\zeta(k-3,3)$ and $\zeta(k)$ satisfy $\dim_\C M_k(\Gamma_1)$ relations over $\Q$, 
where $M_k(\Gamma_1)$ denotes the $\C$-vector space of modular forms of weight $k$ on $\Gamma_1$. 
Broadhurst and Kreimer \cite{BordhurstKreimer97} expanded upon Zagier's observation, extending it to higher depths, and then, Brown \cite{Brown14,Brown21} conducted further investigations, proposing more profound conjectural links between multiple zeta values and modular forms.
These conjectures also hint at modular relations for higher depths, but still remain open.
Some of the results in this direction (i.e.~depth $\ge3$) can be found in \cite{CarrGanglSchneps15,DMNW,EnriquezLochak16,Goncharov01,Li19,MaTasaka21,Tasaka16,Tasaka22} (we do not intend to present these works in this note).

A formulation of modular relations for double zeta values was first established by Gangl, Kaneko and Zagier \cite{GanglKanekoZagier06}.
It uses the space of even period polynomials, which by the theory of Eichler-Shimura is isomorphic to the space of modular forms on $\Gamma_1$.
The first example of their modular relations is simplified to the form
\begin{equation} \label{eq:GKZ12} 
28\zeta(3,9) + 150 \zeta(5,7) + 168\zeta(7,5) = \frac{5197}{691}\zeta(12).
\end{equation}
A key fact is that the above coefficients $28,150,168$ are derived from the cusp form $\Delta(z)=q\prod_{n\ge1}(1-q^n)^{24}$ of weight 12 on $\Gamma_1$ (later, we will describe how these coefficients are obtained).
The significance of their result is as follows.
In the purely algebraic setting (more precisely, using formal double zeta values), their modular relations in weight $k$ not only capture all linear relations among $\zeta(k)$ and $\zeta(odd_{\ge1},odd_{\ge3})$'s of weight $k$, but also can be regarded as another characterization of even period polynomials (and modular forms of weight $k$).
Due to these significant contributions to both multiple zeta values and modular forms, this study has received a substantial amount of interest and attention.

\subsection{Contents of this note}
This note provides a catalog focusing on the actual computation of known modular relations, and it is not our purpose to give an overview of proofs (refer to the original paper for proofs).
The contents of this note are as follows.

In \S2, we begin with modular relations for double zeta values $\zeta(odd_{\ge1},odd_{\ge3})$ due to Gangle, Kaneko and Zagier \cite{GanglKanekoZagier06} in 2006.
The relation they obtained is depicted in the formal double zeta space, defined by double shuffle relations for double zeta values.
We also mention Ma's result \cite{Ma16} in 2016, where modular relations for double zeta values $\zeta(odd_{\ge3},even_{\ge2})$ were obtained.
Every modular relation known so far has been formulated using period polynomials.
We call such relations ``period polynomial relations."

In \S3, a simpler formula for the above two modular relations coming from cusp forms (therefore, called a cuspidal relation in this note) due to Ma and the author in \cite{MaTasaka21} will be given.
Its natural lifting to double Eisenstein series obtained in \cite{Tasaka20} is also described.
I believe these results are worth discussing in the subsequent section where generalizations of period polynomial relations for various kinds of double zeta values are presented.

\S4 collects all results on period polynomial relations known so far.
Hirose \cite{Hirose23} explored further analogues to modular relations for double zeta values, using shuffle regularized double zeta values, consequently yielding several novel results.
Some of his results deal with period polynomials for congruence subgroups of level 2, while the double zeta values he uses are of level 1.
There are several variations on the multiple zeta values of level $N$.
Using one of them, Bachmann \cite{Bachmann20} derived a level 2 analogue of the modular relations established by Gangl, Kaneko and Zagier.
In contrast to Hirose, Bachmann only uses even period polynomials of level 1.
For $N\ge3$, we will also review modular relations for double zeta values of level $N$ recently developed by Kaneko-Tsumura \cite{KanekoTsumura} ($N=4$) and Hirose \cite{Hirose} ($N$: general).

It should be noted that there are several studies on the dual counterparts of modular relations, such as the Ihara-Takao relation, which are relatively easy (from experience) and even helpful to observe a correspondence with period polynomials (cf.~\cite{BaumardSchneps13,Hirose23,MaTasaka21,Tasaka16}).
Due to limitations in space and time, this note does not delve into these results.

\subsection{Notations}
Some standard notations we use are summarized below.
Let $\N$ be the set of a positive integer.
For $N\in \N$, the principal congruence subgroup of level $N$ is
\[ \Gamma(N) :=\left\{ \begin{pmatrix}a&b\\c&d \end{pmatrix} \in \Gamma_1 \ \middle| \ \begin{pmatrix}a&b\\c&d \end{pmatrix} \equiv \begin{pmatrix}1&0\\0&1 \end{pmatrix} \bmod{N}\right\}.\]
The following congruence subgroups of level $N$ (i.e. a subgroup of $\Gamma_1$ which contains $\Gamma(N)$) will appear:
\begin{align*}
 \Gamma_0(N) &:=\left\{ \begin{pmatrix}a&b\\c&d \end{pmatrix} \in \Gamma_1 \ \middle| \ \begin{pmatrix}a&b\\c&d \end{pmatrix} \equiv \begin{pmatrix}\ast&\ast \\0 &\ast \end{pmatrix} \bmod{N}\right\},\\
 \Gamma_1(N) &:=\left\{ \begin{pmatrix}a&b\\c&d \end{pmatrix} \in \Gamma_1 \ \middle| \ \begin{pmatrix}a&b\\c&d \end{pmatrix} \equiv \begin{pmatrix}1&\ast\\ 0&1 \end{pmatrix} \bmod{N}\right\},
\end{align*}
where ``$\ast$" means ``unspecified."
For such subgroup $\Gamma\subset \Gamma_1$, we denote by $M_k(\Gamma)$ and (resp.~$S_k(\Gamma)$) the $\C$-vector space of holomorphic modular (resp.~cusp) forms of weight $k$ on $\Gamma$.
These are finite-dimensional vector spaces.
For example, for $k\ge2$ even we have
\[ \dim_\C M_k(\Gamma_1) = \left[\frac{k}{4}\right]-\left[\frac{k-2}{6}\right],\]
and $\dim_\C S_k(\Gamma_1)=\dim_\C M_k(\Gamma_1)-1$.
For the dimension formula, see e.g.~\cite{Serre} for $N=1$ and \cite{DS} for $N$ general.

\section{Period polynomial relations for double zeta values}

\subsection{Period polynomials}

Let us recall the theory of period polynomials (cf.~\cite[\S5]{GanglKanekoZagier06}).

For $w\ge0$ even, denote by $V_w$ the $\Q$-vector space spanned by $X^aY^{w-a} \ (0\le a\le w)$.
For $P(X,Y)\in V_w$, we write
\[ \big(P\big| \gamma\big) (X,Y):= P(aX+bY,cX+dY), \quad \forall \gamma=\begin{pmatrix} a&b\\ c&d\end{pmatrix} \in {\rm PGL}_2(\Z),\]
where ${\rm PGL}_2(\Z):={\rm GL}_2(\Z)/\{\pm I\}$.
This action of ${\rm PGL}_2(\Z)$ on $V_w$ is extended to an action of the group ring $\Z[{\rm PGL}_2(\Z)]$ by linearity.
As usual, we set
\[ S=\begin{pmatrix} 0&-1\\1&0\end{pmatrix}, \  T=\begin{pmatrix} 1&1\\0&1\end{pmatrix}, \ U=TS=\begin{pmatrix} 1&-1\\1&0\end{pmatrix},\  \varepsilon=\begin{pmatrix} 0&1\\1&0\end{pmatrix}, \  \delta=\begin{pmatrix} -1&0\\0&1\end{pmatrix}.\]
The space of period polynomials is then defined by
\[ W_w := \{P\in V_w \mid P\big| (1+S)=P\big|(1+U+U^2)=0\}.\]

The group ${\rm PSL}_2(\Z):=\Gamma_1/\{\pm I\}$ is generated by $S$ and $T$ and satisfies the relations $S^2=U^3=I$.
Since we have $S=\varepsilon \delta=\delta\varepsilon$ and $\varepsilon U\varepsilon = U^2$, the space $W_w$ has the direct sum decomposition 
\[ W_w=W_w^+\oplus W_w^-,\quad \mbox{where}\quad W_w^{\pm}:=\{P\in W_w\mid P\big| \delta=\pm P\}.\]
By linear algebra, for $w\ge0$ even, we have 
\[\dim_\Q W_{w}^{+} = \left[\frac{w+2}{4}\right]-\left[\frac{w}{6}\right]\quad \mbox{and}\quad \dim_\Q W_{w}^{-} = \dim_\Q W_{w}^{+}-1.\] 
Note that the space $W_w^+$ of even period polynomials always contains the polynomial $X^w-Y^w$ and has the decomposition $\Q(X^w-Y^w)\oplus W_w^{+,0}$, where $W_w^{+,0}=\{P\in W_w^+\mid P(X,0)=0\}$.
For example, when $w=10$, the following polynomials are a basis of each space:
\begin{align*}
 &X^2Y^2(X^2-Y^2)^3 \in W_{10}^{+,0},\qquad XY (X^2-Y^2)^2 (4X^4-17 X^2Y^2+4Y^4)\in W_{10}^-.
\end{align*}

We now relate the period polynomials with cusp forms.
This is done by the Eichler-Shimura theory of periods of cusp forms (cf.~\cite{Lang76}).
For a cusp form $f\in S_k(\Gamma_1)$, let
\begin{align*}
P_f(X,Y)&:=\int_0^{i\infty} f(z)(X-Yz)^{k-2}dz=\sum_{\substack{r+s=k\\r,s\ge1}}(-1)^{s-1} \binom{k-2}{s-1} \Lambda(f;s) X^{r-1}Y^{s-1},
\end{align*}
where we set $\Lambda(f;s) :=\int_0^{i\infty} f(z)z^{s-1}dz$.
The values of $\Lambda(f;s) $ at $s=1,2,\ldots,k-1$ are called {\itshape critical}.
Using the transformation formula $\big(P_f\big|\gamma\big) (X,Y)=\int_{\gamma^{-1}(0)}^{\gamma^{-1}(i\infty)} f(z)(X-Yz)^{k-2}dz$ for $\gamma\in \Gamma_1$, one easily finds that $P_f\in W_{k-2}\otimes_\Q \C$ for any $f\in S_k(\Gamma_1)$, which implies that the critical values of $f$ satisfy many $\Q$-linear relations.
The Eichler-Shimura theory tells us that the linear maps 
\[ r^\pm : S_k(\Gamma_1) \longrightarrow W_{k-2}^\pm\otimes_\Q \C, \qquad f\longmapsto \frac12 P_f\big|(1\pm \delta)=:P_f^\pm (X,Y)\]
are injective. 
In particular, the map $r^-$ is an isomorphism due to the dimension formula:
\begin{align*}
 \dim_\C S_k(\Gamma_1)&=\dim_\Q W_{k-2}^-,\quad \dim_\C M_k(\Gamma_1)=\dim_\Q W_{k-2}^+.
\end{align*}

\begin{example}
Let us illustrate the period polynomial of the discriminant function $\Delta(z)=q\prod_{n\ge1}(1-q^n)^{24}\in S_{12}(\Gamma_1)$.
We have seen that the set $\{X^{10}-Y^{10},X^2Y^2(X^2-Y^2)^3\}$ forms a basis of $W_{10}^+$, so the polynomial $P_\Delta^+$ is a linear combination of them.
For the explicit description, we use Manin's result \cite{Manin73} on the ratio of critical values (e.g.~$\frac{\Lambda(\Delta;1)}{\binom{10}{2}\Lambda(\Delta;3)}=-\frac{36}{691}$
) to get
\begin{align}
\label{eq:ev_delta}
P_\Delta^+(X,Y)&=\binom{10}{2} \Lambda(\Delta;3)  \left\{ -\frac{36}{691}(X^{10}-Y^{10})+X^2Y^2(X^2-Y^2)^3\right\} .
\end{align}
Similarly, one can verify the identity
\begin{align}
\label{eq:od_delta}
P_\Delta^-(X,Y)&=-\binom{10}{1}\Lambda(\Delta;2) \ XY (X^2-Y^2)^2 \left(X^4-\frac{17}{4} X^2Y^2+Y^4\right).
\end{align}
\end{example}

There are generalizations of the maps $r^\pm$ for modular forms.
What we need is to extend the period polynomials of cusp forms to the Eisenstein series, because every modular forms is a linear combination of the Eisenstein series and cusp forms.
For one, see Zagier's extension \cite[\S2]{Zagier91}.
Another construction using $1$-cocycles was given by Brown \cite[\S7]{Brown14}.

\begin{remark} 
We make a remark on a calculation of period polynomials.
For $w\ge0$ even, let 
\[V_w^\pm:=\{P\in V_w\mid P\big|\delta=\pm P\}.\]
The defining equations of $W_w^\pm$ can be simplified to one equation in $V_w^\pm$.
We only recall the case $W_w^+$ (see \cite[\S5]{GanglKanekoZagier06} for the details):
\[ W_w^+=\{P\in V_w^+\mid P(X,Y)-P(Y-X,Y)+P(Y-X,X)=0\}.\]
A straightforward computation shows that, for an even polynomial $P\in V_{k-2}^+$, it lies in $W_{k-2}^+$ if and only if the coefficients $\sum_{r+s=k}a_{r,s}X^{r-1}Y^{s-1}=P$ satisfy
\[ \sum_{\substack{h+p=k\\h,p\ge1}} a_{h,p} C^p_{r,s}=0\qquad \mbox{for all $r,s\ge1$ with $r+s=k$},\]
where for positive integers $r,s,p$, we let
\begin{equation}\label{eq:C} 
C^p_{r,s} := \delta_{p,r} + (-1)^r \binom{p-1}{r-1}+(-1)^{p-s}\binom{p-1}{s-1}
\end{equation}
and $\delta_{p,r}$ is the Kronecker delta.
Notably, it was pointed out in \cite[Proposition 3.2]{BaumardSchneps13} (see also \cite[Proposition 3.4]{Tasaka16}) that the latter condition is reduced to ``for all $r,s\ge1$ odd with $r+s=k$."
To get a cuspidal polynomial $P$ in $W_{k-2}^{+}$, i.e. an element in the image $r^+\big(S_k(\Gamma_1)\big)$, we further impose the relation 
\[  \sum_{\substack{r+s=k\\r,s\ge1:{\rm odd}}} (r-1)!(s-1)! \lambda(r,s) a_{r,s}=0 ,\]
where for $k\ge3,r,s\ge1$ with $r+s=k$, we set $\beta(k) := -\frac{B_k}{2k!}$, where $B_k$ is the $k$th Bernoulli number, and
\begin{equation}\label{eq:beta} 
\begin{aligned}
\lambda(r,s):=-\frac{\beta(k)}{12} \left( 1 -(-1)^{s} \binom{k-1}{s-1}+(-1)^{s} \binom{k-1}{s} \right)-\frac{(-1)^{s}}{3} \sum_{j=2}^k \binom{j-1}{s-1} \beta(j) \beta(k-j).
\end{aligned}
\end{equation}
Indeed, for $f\in S_k(\Gamma_1)$, the coefficients $\sum_{r+s=k}a_{r,s}X^{r-1}Y^{s-1}=P_f^+(X,Y)$  satisfy the above relation, which was first proved by Kohnen and Zagier \cite[Theorem 9 (ii)]{KohnenZagier84} as a consequence of Haberland's formula for the Petersson inner product between the Eisenstein series and cusp forms.
Here our $\lambda(k-s,s)$ coincides with $\frac{\lambda_{k/2,s-1}}{24 k!}$, where $\lambda_{k,n}$ is given in \cite[Theorem 9]{KohnenZagier84}.
\end{remark}

\subsection{Double shuffle relations}
The double shuffle relation is a family of linear relations among multiple zeta values, resulting from two different representations of the product of two multiple zeta values.
We will use the following simplest case, which is also known as a special case of the regularized double shuffle relation (cf.~\cite{IharaKanekoZagier06}).

\begin{proposition}\label{prop:dsr}
For positive integers $r,s\ge2$, we have
\[\zeta(r,s)+\zeta(s,r)+\zeta(r+s) = \sum_{\substack{h+p=r+s\\p\ge\min\{r,s\}}} \left(\binom{p-1}{r-1}+\binom{p-1}{s-1}\right) \zeta(h,p).\]
Moreover, if we set $\zeta(r,1):=0$ for $r\ge2$, then the above relation also holds for $r,s\ge1$, $r+s\ge3$.
\end{proposition}
\begin{proof}
We compute $\zeta(r)\zeta(s)=\sum_{m,n>0}m^{-r} n^{-s}$ in two ways.
Firstly, the left-hand side is obtained by decomposing $\N^2$ into three domains $\{(m,n)\in \N^2\mid m<n\}$, $\{(m,n)\in \N^2\mid m>n\}$ and $\{(m,n)\in \N^2\mid m=n\}$.
Secondly, using the partial fractional expansion
\begin{equation}\label{eq:pfe} 
\frac{1}{m^rn^s} =  \sum_{\substack{h+p=r+s\\h,p\ge1}} \left(\frac{\binom{p-1}{r-1}}{n^h(m+n)^p}+\frac{\binom{p-1}{s-1}}{m^h (m+n)^p} \right),
\end{equation}
and then, summing over $m,n\ge1$ on both sides, we get the right-hand side.
The double shuffle relation for the case $r\ge2$ and $s=1$ is the same as Euler's sum formula $\zeta(r+1)=\sum_{p=2}^r \zeta(r+1-p,p)$, which is well-known (from the theory of the regularized double shuffle relation, one may set $\zeta(r,1)$ to be $-\zeta(1,r)-\zeta(r+1)$ for instance, but the obtained relation does not depend on the choice of $\zeta(r,1)$ except for ``$\zeta(r,1)=\infty$").
\end{proof}

Modular relations for double zeta values become more significant when formulated in terms of the formal double zeta space, introduced by Gangl, Kaneko and Zagier \cite{GanglKanekoZagier06}.
For each $k\in \N$, the formal double zeta space $\mathcal{D}_k$ is defined to be the $\Q$-vector space spanned by symbols $Z_k,Z_{r,s} \ (r+s=k, r\ge1,s\ge2)$ which are subject to the relations 
\[Z_{r,s}+Z_{s,r}+Z_{k} = \sum_{h+p=k} \left(\binom{p-1}{r-1}+\binom{p-1}{s-1}\right) Z_{h,p}\]
for all $r,s\ge1$ with $r+s=k$.
Hereafter, we use the convention that the condition ``$h+p=k$" (resp.~``$r+s=k$") includes ``$h,p\ge1$" (resp.~``$r,s\ge1$").
Since the above relations for the cases $r\ge s\ge 1$ are linearly independent, for each $k\ge3$ it follows that 
\[ \dim_\Q \mathcal{D}_k=\left[ \frac{k-1}{2}\right].\] 

By definition, there is the natural surjective linear map from $\mathcal{D}_k$ to the $\Q$-vector space $\mathcal{DZ}_k$ spanned by $\zeta(k)$ and all double zeta values of weight $k$.
This map is not injective in general when $k$ is even; According to Zagier's dimension conjecture \cite{Zagier94,Zagier93} stating that
\begin{equation}\label{eq:Zagier_conj} 
\dim_\Q \mathcal{D}_k - \dim_\Q \mathcal{DZ}_k \stackrel{?}{=} \dim_\C S_k (\Gamma_1) \quad (k\ge4),
\end{equation}
we might expect that the kernel of the map $\mathcal{D}_k\rightarrow \mathcal{DZ}_k$ is related to cusp forms on $\Gamma_1$.
\begin{center}
\begin{tabular}{c|cccccccccccccccc}
$k$ & 4&6& 8 & 10 & 12 & 14 & 16 & 18 & 20 & 22 & 24 & 26\\ \hline
$\left[ \frac{k-1}{2}\right] $ & 1 & 2 & 3 & 4 & 5 & 6 & 7 & 8 & 9 & 10 & 11 & 12 \\ \hline
$\dim S_{k}(\Gamma_1) $ & 0 & 0 & 0 & 0 & 1 & 0 & 1 & 1 & 1 & 1 & 2 &1\\ \hline
\end{tabular}
\end{center}

Let us elaborate on the conjecture \eqref{eq:Zagier_conj} when $k$ is even. 
It was shown in \cite[Theorem 2]{GanglKanekoZagier06} that every $Z_{even_{\ge2},even_{\ge2}}$ of weight $k$ is a $\Q$-linear combination of $Z_{odd_{\ge1},odd_{\ge3}}$'s and $Z_k$, where $Z_{even_{\ge2},even_{\ge2}}$ of weight $k$ means $Z_{r,s}$ with $r,s\ge2$ even and $r+s=k$.
Counting these generators, together with the restricted sum formula explained in Example \ref{ex:restricted_sum} below, we see that the set $\{Z_{k-s,s} \mid 3\le s\le k-1:{\rm odd}\}$ forms a basis of $\mathcal{D}_k$.
Therefore, the equality \eqref{eq:Zagier_conj} would show that the values $\zeta(odd_{\ge1},odd_{\ge3})$ of weight $k$ satisfy $\dim_\C S_k(\Gamma_1)$ relations over $\Q$.
Notably, this was first pointed out by Zagier in \cite{Zagier94}.

We also mention the case $k$ odd.
In this case, the equality \eqref{eq:Zagier_conj} implies that every linear relation of double zeta values is obtained from Proposition \ref{prop:dsr}.
We should have provided a clearer explanation, but an affirmative answer to this speculation is already indicated by Zagier in \cite[Theorem 2]{Zagier12}.

\subsection{Modular relation for even weight double zeta values}\label{subsec:GKZ}

In \cite{GanglKanekoZagier06}, the modular relation for double zeta values is stated as a special type of relations in the formal double zeta space $\mathcal{D}_k$.

\begin{theorem}\label{thm:GKZ}
Let $k\ge4$ even. 
For $P\in W_{k-2}^+$, define $a_{r,s} \in \Q \ (r,s\ge1)$ by
\[ \sum_{r+s=k} \binom{k-2}{r-1}a_{r,s}X^{r-1}Y^{s-1}:=P(X+Y,X).\]
Then, we have $a_{r,s}=a_{s,r} \ (r,s:$ even$)$, $a_{k-1,1}=0$ and 
\[ 3\sum_{\substack{r+s=k\\ r,s:{\rm odd}} } a_{r,s}Z_{r,s}= \sum_{\substack{r+s=k\\ r,s:{\rm even}} } a_{r,s}Z_{r,s} + \left(\sum_{r+s=k} (-1)^r a_{r,s}\right) Z_{k} .\]
Here we use the convention that the summation variables $r,s$ are $\ge1$.
\end{theorem}

\begin{example}\label{ex:restricted_sum}
The special case $P=X^{k-2}-Y^{k-2}$ in Theorem \ref{thm:GKZ} yields the relation
\begin{equation*}\label{eq:sum_12}
\begin{aligned}
&3 \sum_{\substack{1\le r\le k-3\\ r:{\rm odd}}} Z_{r,k-r}=\sum_{\substack{2\le r\le k-2\\ r:{\rm even}}} Z_{r,k-r}.
\end{aligned}
\end{equation*}
By Euler's sum formula $\sum_{r=1}^{k-2}Z_{r,k-r}=Z_k$ for $k\ge3$, the above relation is reduced to the ``restricted sum formulas"
\[ \sum_{\substack{1\le r\le k-3\\ r:{\rm odd}}} Z_{r,k-r}=\frac14 Z_k\quad \mbox{and}\quad \sum_{\substack{2\le r\le k-2\\ r:{\rm even}}} Z_{r,k-r}=\frac34 Z_{k}, \qquad (k\ge4: {\rm even}).\]
\end{example}

\begin{example}\label{ex:GKZ12}
Taking $P=145110\times \big(-\frac{36}{691}(X^{10}-Y^{10})+X^2Y^2(X^2-Y^2)^3\big)$, which by \eqref{eq:ev_delta} is proportional to the period polynomial $P_\Delta^+$ of the cusp form $\Delta(z)$, in Theorem \ref{thm:GKZ} gives
\begin{equation}\label{eq:GKZ_delta}
\begin{aligned}
&22680 Z_{1,11}+13006 Z_{3,9}-29145 Z_{5,7}-35364 Z_{7,5}+22680 Z_{9,3}\\
&=7560 Z_{2,10}-2114 Z_{4,8}-\frac{42965}{3} Z_{6,6}-2114 Z_{8,4}+7560 Z_{10,2}-1382 Z_{12}.
\end{aligned}
\end{equation}
Note that \eqref{eq:GKZ12} can be derived from \eqref{eq:GKZ_delta} under the correspondence $Z\mapsto \zeta$.
To see this, we need to use the harmonic product formula $\zeta(r,s)+\zeta(s,r)=\zeta(r+s)-\zeta(r)\zeta(s) $ for $r,s\in \Z_{\ge2}$ and Euler's formula $\zeta(2k)=-\frac{(-1)^{k+1}B_{2k}(2\pi )^{2k}}{2(2k)!} $ for $k\in \Z_{\ge1}$, which does not hold in $\mathcal{D}_{2k}$, where $B_{2k}$ is the $2k$th Bernoulli number.
Later, we will discuss a simplification of the right-hand side of \eqref{eq:GKZ_delta}, though this simplified relation cannot be true in the formal double zeta space.
\end{example}

Theorem \ref{thm:GKZ} says that the values $\zeta(k)$ and $\zeta(odd_{\ge1},odd_{\ge3})$'s of weight $k$ satisfy $\dim_\Q W_{k-2}^+$ relations over $\Q$. 
Hence the inequality $\dim_\Q \mathcal{D}_k - \dim_\Q \mathcal{DZ}_k \ge\dim_\C S_k (\Gamma_1) $ is valid for any $k\ge 4$ even.
This gives an affirmative answer to \eqref{eq:Zagier_conj}, because proving $\le$ (i.e.~linear independence of multiple zeta values) is a hard problem at present.

It is worth mentioning that the original statement in \cite[Theorem 3]{GanglKanekoZagier06} includes the opposite implication that provides another characterization of even period polynomials via regularized double shuffle relations of double zeta values. 

\subsection{Modular relation for odd weight double zeta values}

In \cite[Theorem 3]{Zagier12}, for each odd integer $k\ge5$, Zagier showed that the values $\zeta(odd_{\ge1},even_{\ge2})$'s of weight $k$ satisfy $[(k-5)/6]$ relations over $\Q$.
Since $[(k-5)/6] = \dim_\C S_{k-1}(\Gamma_1)+\dim_\C S_{k+1}(\Gamma_1)$ for $k\ge5$ odd, they are expected to be related to cusp forms in two ways.
This expectation was revealed by Ma \cite[Theorems 1 and 2]{Ma16}.


\begin{theorem}\label{thm:Ma}
Let $k\ge12$ even.
\begin{itemize}
\item[{\rm (i)}] For $P\in W_{k-2}^-$, let 
\[ \sum_{r+s=k}\binom{k-1}{r-1}b_{r,s}X^{r-1}Y^{s-1}:=P(X+Y,Y)-\frac{X}{Y}P(X+Y,X) .\]
Then we have $b_{1,k-1}=b_{k-1,1}=0$ and
\[ \sum_{\substack{r+s=k\\r,s:{\rm odd}}} b_{r,s}Z_{r,s+1}\equiv 0\mod \Q Z_{k+1} .\]
\item[{\rm (ii)}] For $P\in W_{k-2}^{+}$, let
\[ \sum_{r+s=k-1}\binom{k-3}{r-1}c_{r,s}X^{r-1}Y^{s-1}:=\frac{d}{dX}P(X+Y,Y)-\frac{d}{dY}P(X+Y,X) . \]
Then we have $c_{1,k-2}=c_{k-3,2}=0$ and
\[ \sum_{\substack{r+s=k-1\\r:{\rm odd}}} c_{r,s}Z_{r,s}\equiv 0\mod \Q Z_{k-1} .\]
\end{itemize}
\end{theorem}

Notice that taking $P=X^{k-2}-Y^{k-2}\in W_{k-2}^+$ in Theorem \ref{thm:Ma} (ii), we get $c_{r,s}=0$ for all $r+s=k-1$, so only the subspace $W_{k-2}^{+,0}$ provides non-trivial relations.

\begin{example}
The special case $P=X^2Y^2(X^2-Y^2)^3 \in W_{10}^{+,0}$ in Theorem \ref{thm:Ma} (ii) and the special case $P=XY (X^2-Y^2)^2 (4X^4-17 X^2Y^2+4Y^4)\in W_{10}^-$ 
in Theorem \ref{thm:Ma} (i) give
\begin{align}
\label{eq:odd_11} &14Z_{3,8} + 10 Z_{5,6} - 21Z_{7,4} = -\frac32 Z_{11}, \\
\label{eq:odd_13} &12Z_{3,10} + 14 Z_{5,8} - 5 Z_{7,6} -18 Z_{9,4}= -\frac32 Z_{13} ,
\end{align}
respectively.
Here the coefficients of $Z_{r,s}$'s in the above relations are normalized so that they are coprime integers.
\end{example}

As a consequence of Theorem \ref{thm:Ma}, we see that double zeta values $\zeta(odd_{\ge3},even_{\ge4})$ of weight $k$ and $\zeta(k)$ satisfy $\dim_\C S_{k+1}(\Gamma_1)$ relations (resp.~$\dim_\C S_{k-1}(\Gamma_1)$ relations), whose coefficients are obtained from odd period polynomials (resp.~`derivative' of even period polynomials).
Since two relations obtained from Theorem \ref{thm:Ma} (i) and (ii) are independent, for each $k\ge5$ odd, one has
\[ \dim_\Q \langle \zeta(k),\zeta(odd_{\ge3},even_{\ge4}) \ \mbox{of weight $k$}\rangle_\Q\le \frac{k-3}{2}  - \dim_\C S_{k+1}(\Gamma_1) -\dim_\C S_{k-1}(\Gamma_1).\]
\begin{center}
\begin{tabular}{c|cccccccccccccccc}
$k$ & 5&7& 9 & 11 & 13 & 15 & 17 & 19 & 21 & 23 & 25 & 27\\ \hline
$\frac{k-3}{2} $ & 1 & 2 & 3 & 4 & 5 & 6 & 7 & 8 & 9 & 10 & 11 & 12 \\ \hline
$\dim S_{k+1}(\Gamma_1) $ & 0 & 0 & 0 & 1 & 0 & 1 & 1 & 1 & 1 & 1 & 2 &1\\ \hline
$\dim S_{k-1}(\Gamma_1) $ & 0 & 0 & 0 & 0 & 1 & 0 & 1 & 1 & 1 & 1 & 1 & 2 \\ \hline
\end{tabular}
\end{center}

Using the motivic set-up, Li and Liu \cite{LiLiu} proved that Theorem \ref{thm:Ma} provides all $\Q$-linear relations among $Z_{k}$ and $Z_{odd_{\ge3},even_{\ge2}}$'s of weight $k\ge5$ odd.

\

\begin{problem}
Can we characterize odd period polynomials by regularized double shuffle relations?
For example, do the relation in $\mathcal{D}_{k+1}$ of the form $\sum_{\substack{r+s=k\\r,s:{\rm odd}}} b_{r,s}Z_{r,s+1}\equiv 0\mod \Q Z_{k+1}$ with $b_{1,k-1}=b_{k-1,1}$ characterize all odd period polynomial of degree $k-2$?
\end{problem}

\section{Cuspidal relations for double zeta values}
\subsection{Cuspidal part of modular relations}
Recall that the coefficient of the single zeta value in the relation of both Theorem \ref{thm:GKZ} and Theorem \ref{thm:Ma} under the correspondence $Z\mapsto \zeta$ is computable from their results (see the discussion in Example \ref{ex:GKZ12}).
A simple formula for the coefficient was given in \cite{MaTasaka21} when the corresponding period polynomial is cuspidal.
Surprisingly, this is done by just replacing double zeta values $\zeta(r,s)$ with Yamamoto's $\frac12$-interpolated double zeta values $\zeta^{\frac12}(r,s):=\zeta(r,s)+\frac12 \zeta(r+s)$ (see \cite{Yamamoto}).

\begin{theorem}\label{thm:Refinement}
For a cusp form $f\in S_k(\Gamma_1)$, define numbers $a_{r,s},b_{r,s},c_{r,s}$ by
\begin{align*}
&\sum_{r+s=k} \binom{k-2}{r-1}a_{r,s}^f X^{r-1}Y^{s-1}:=P_f^+(X+Y,X),\\
&\sum_{r+s=k}\binom{k-1}{r-1}b_{r,s}^f X^{r-1}Y^{s-1}:=P_f^-(X+Y,Y)-\frac{X}{Y}P_f^-(X+Y,X),\\
&\sum_{r+s=k-1}\binom{k-3}{r-1}c_{r,s}^f X^{r-1}Y^{s-1}:=\frac{d}{dX}P_f^+(X+Y,Y)-\frac{d}{dY}P_f^+(X+Y,X) .
\end{align*}
Then we have $a_{r,s}^f=a_{s,r}^f \ (r,s:$ even$)$, $a_{k-1,1}^f=0$, $b_{1,k-1}^f=b_{k-1,1}^f=0$, $c_{1,k-2}^f=c_{k-3,2}^f=0$ and
\begin{align*}
&\sum_{\substack{r+s=k\\ r,s:{\rm odd}} } a_{r,s}^f\zeta^{\frac12}(r,s)=0,\ 
\sum_{\substack{r+s=k\\r,s:{\rm odd}}} b_{r,s}^f\zeta^{\frac12}(r,s+1)= 0,\
\sum_{\substack{r+s=k-1\\r:{\rm odd}}} c_{r,s}^f\zeta^{\frac12}(r,s)=0 .
\end{align*}
\end{theorem}

Kohnen and Zagier \cite{KohnenZagier84} showed that there is a basis of the space of cusp forms such that their critical values (at the same parity) are rational numbers.
Hence the above relations can be over $\Q$.

\begin{example}
Taking $f=\Delta(z)$ in Theorem \ref{thm:Refinement} and then multiplying by certain constants, we get the identities
\begin{align}
\label{eq:cuspidal_rel}0=&22680 \zeta^{\frac12}(1,11)+13006 \zeta^{\frac12}(3,9)-29145 \zeta^{\frac12}(5,7)-35364 \zeta^{\frac12}(7,5)+22680 \zeta^{\frac12}(9,3),\\
\notag 0=&14\zeta^{\frac12}(3,8) + 10 \zeta^{\frac12}(5,6) - 21\zeta^{\frac12}(7,4) ,\\
\notag 0=&12\zeta^{\frac12}(3,10) + 14\zeta^{\frac12}(5,8) - 5\zeta^{\frac12}(7,6) -18\zeta^{\frac12}(9,4),
\end{align}
where the last two identities are equivalent to \eqref{eq:odd_11} and \eqref{eq:odd_13} (and hold in the formal double zeta space, replacing $\zeta$ with $Z$), but the first identity needs an extra work to get \eqref{eq:GKZ_delta} (and does not hold in $\mathcal{D}_{12}$).
As another example, using the unique cusp form $f= q+216q^2-3348q^3+\cdots \in S_{16}(\Gamma_1)$, we can get
\begin{align*}
0=&1081080 \zeta^{\frac12}(1,15)+842358 \zeta^{\frac12}(3,13)-275295 \zeta^{\frac12}(5,11)-1400182
   \zeta^{\frac12}(7,9)\\
   &-1360395 \zeta^{\frac12}(9,7)-351252 \zeta^{\frac12}(11,5)+1081080\zeta^{\frac12}(13,3) ,\\
0=&22 \zeta^{\frac12}(3,12)+30 \zeta^{\frac12}(5,10)+7 \zeta^{\frac12}(7,8)-20 \zeta^{\frac12}(9,6)-33
   \zeta^{\frac12}(11,4),\\
0=&156 \zeta^{\frac12}(3,14)+242 \zeta^{\frac12}(5,12)+153 \zeta^{\frac12}(7,10)-56 \zeta^{\frac12}(9,8)-215  \zeta^{\frac12}(11,6)-234 \zeta^{\frac12}(13,4).
\end{align*}
\end{example}

\begin{remark}
Our proof of the first identity in Theorem \ref{thm:Refinement} relies heavily on the theory of motivic multiple zeta values developed by Brown \cite{Brown12} (see \cite{MaTasaka21} for the details). 
As a byproduct of our proof, we found the coincidence between a rational solution to the double shuffle equation of depth 2 due to Gangl, Kaneko and Zagier \cite{GanglKanekoZagier06} and the coefficients $\lambda(r,s)$ defined in \eqref{eq:beta} in the extra relation of critical values of a cusp form due to Kohnen and Zagier \cite{KohnenZagier84}. 
More precisely, for $r,s\ge1$, $r+s=k$ even, let 
\[\beta(r,s):=\beta(r)\beta(s) - \frac{\beta(k)}{2} - \lambda(r,s) .\]
Then $\beta$ coincides with the Bernoulli realization of $\mathcal{D}_k$ given in \cite{GanglKanekoZagier06}, which is one of solutions to the double shuffle equation of depth 2.
Namely, $\beta(r,s)+\beta(s,r)+\beta(k) = \sum_{h+p=k} \left(\binom{p-1}{r-1}+\binom{p-1}{s-1}\right) \beta(h,p)$ holds for any $r,s\ge1$, $r+s=k\ge3$.
Note that other solutions to the double shuffle equation of depth 2 were found by Brown \cite{Brown17} and by Ma and the author \cite{MaTasaka21}.
These solutions are unique up to solutions to the linearized double shuffle equation modulo product of depth 2.
\end{remark}

\subsection{Double Eisenstein series interpretation}
The concept of multiple Eisenstein series was first introduced by Gangl, Kaneko and Zagier \cite{GanglKanekoZagier06}.
It is a holomorphic function on the upper half-plane having the Fourier expansion whose constant term is a multiple zeta value.
Since the space spanned by multiple Eisenstein series contains the space of modular forms on $\Gamma_1$ (cf.~\cite[Theorem 5]{GanglKanekoZagier06}), expressing modular forms in terms of linear combinations of multiple Eisenstein series and then comparing the $q$-expansions, we get linear relations of multiple zeta values.
This machinery gives a naive exposition why modular forms yield linear relations of multiple zeta values.
The purpose of this section is to understand Theorem \ref{thm:GKZ} (or rather, Theorem \ref{thm:Refinement}) from this viewpoint.

For $k_1,\ldots,k_{d-1}\in \Z_{\ge2}$ and $k_d\in\Z_{\ge3}$ (for convergence) and $z$ in the complex upper half-plane, the multiple Eisenstein series $G_{k_1,\ldots,k_d}(z)$ is defined by
\[G_{k_1,\ldots,k_d}(z) := \sum_{\substack{0\prec \lambda_1\prec \cdots\prec \lambda_d\\ \lambda_i\in\Z z+\Z}} \frac{1}{\lambda_1^{k_1}\cdots \lambda_d^{k_d}},\]
where a lattice point $\lambda=\ell z+m \in \Z z+\Z$ on the complex plain is said to be positive (denoted by $0\prec\lambda$) if either $\ell>0$ or $\ell=0$ and $m>0$ holds.
Then, for $\ell_1z+m_1,\ell_2 z+m_2\in  \Z z+\Z$, we define the order $\ell_1z+m_1\prec \ell_2z+m_2$ if $0\prec(\ell_2-\ell_1)z+(m_2-m_1)$.
The case $r=1$ and $k_1\ge4$ even is the classical Eisenstein series $G_{k_1}(z)$, which is modular.
By definition, the multiple Eisenstein series satisfies the harmonic product formula (e.g., $G_r(z)G_s(z)=G_{r,s}(z)+G_{s,r}(z)+G_{r+s}(z)$ for $r,s\in \Z_{\ge3}$).

Let us illustrate the Fourier expansion.
Using the Lipschitz formula 
\[\sum_{m\in \Z} (z+m)^{-k}=\frac{(-2\pi i)^k}{(k-1)!} \sum_{n>0} n^{k-1}q^n \ (k\ge2)\] 
with $q:=e^{2\pi iz}$, one can compute the Fourier expansion of $G_{k}(z)$ for $k\ge3$ as follows.
\begin{align*}
G_k(z)&= \sum_{\substack{\ell =0\\m>0}}\frac{1}{m^k} + \sum_{\substack{\ell>0\\ m\in\Z}} \frac{1}{(\ell z+m)^k}\\
&=\zeta(k) + \frac{(-2\pi i)^{k}}{(k-1)!} \sum_{n>0} \sigma_{k-1}(n) q^n ,
\end{align*}
where $\sigma_{k-1}(n)=\sum_{d\mid n}d^{k-1}$ is the divisor function.
Similarly, in \cite{GanglKanekoZagier06}, they computed the Fourier expansion of double Eisenstein series $G_{r,s}(z)$ for $r\ge2,s\ge3$:
\begin{equation}\label{eq:des}
\begin{aligned}
G_{r,s}(z)& = \zeta(r,s)+ \zeta(r)g_s(z) + g_{r,s}(z)\\
&+ \sum_{\substack{h+p=r+s\\ p\ge\min\{r,s\}}} \left( (-1)^r\binom{p-1}{r-1}+(-1)^{p-s}\binom{p-1}{s-1}\right) \zeta(p)g_h(z),
\end{aligned}
\end{equation}
where for positive integers $k_1,\ldots,k_d$, we define the holomorphic function $g_{k_1,\ldots,k_d}(z)$ on the upper half-plane by
\[g_{k_1,\ldots,k_d}(z) := \frac{(-2\pi i)^{k_1+\cdots+k_d}}{(k_1-1)!\cdots (k_d-1)!} \sum_{\substack{0<\ell_1<\cdots<\ell_d\\ n_1,\ldots,n_d>0}} n_1^{k_1-1}\cdots n_d^{k_d-1} q^{\ell_1n_1+\cdots+\ell_d n_d}.\]

It should be noted here that the above formula \eqref{eq:des} is deeply related to even period polynomials, pointed out by Kaneko in \cite{Kaneko07}.
Indeed, recalling \eqref{eq:C}, we get $G_{r,s}(z)=\zeta(r,s)+g_{r,s}(z) + \sum_{h+p=r+s} C_{r,s}^p \zeta(p)g_h(z)$.
More significantly, the Fourier expansion has an intimate connection with the Goncharov coproduct on the space of formal iterated integrals (see \cite{BachmannTasaka18}).

The concept of a `regularization' cam be applied to multiple Eisenstein series.
The above formulas for the Fourier expansion converge rapidly and give holomorphic functions of $z$ for all $k_1,\ldots,k_d\ge2$, to which the definition of $G_{k_1,\ldots,k_d}(z)$ can be slightly extended (we can ensure that they still satisfy the harmonic product formula).
Moreover, there are extensions of $G_{k_1,\ldots,k_d}(z)$ to all positive integers $k_1,\ldots,k_d$ while preserving certain families of relations: the shuffle product formulas \cite{BachmannTasaka18} and the harmonic product formulas \cite{Bachmann19}.
In this section, we use the one given by Gangl, Kaneko and Zagier \cite[\S7]{GanglKanekoZagier06}, which satisfies regularized double shuffle relations for all $r,s\ge1$ with $r+s\ge3$.

For $r,s\ge1$, set
\[ \varepsilon_{r,s}(z):=2\pi i \big( \delta_{s,2}g_r^\ast(z)-\delta_{s,1}g_{r-1}^\ast (z) + \delta_{r,1} \big( g_{s-1}^\ast (z) + g_s(z) \big)\big) + \delta_{r,1}\delta_{s,1} g_2(z),\]
where $g_k^\ast (z) := \frac{-(-2\pi i)^{k+1}}{k!} \sum_{\ell,n>0} \ell n^k q^{\ell n} \ (k\ge0)$, and define
\begin{align*}
G^{\rm reg}_{r,s}(z) &:=\zeta(r,s)+ \zeta(r)g_s(z) + g_{r,s}(z)+\frac12 \varepsilon_{r,s}(z)\\
&+ \sum_{h+p=r+s} \left( (-1)^r\binom{p-1}{r-1}+(-1)^{p-s}\binom{p-1}{s-1}\right) \zeta(p)g_h(z)
\end{align*}
with $\zeta(r,1):=-\zeta(r+1)-\zeta(1,r)$ for $r\ge2$ and $\zeta(1):=0$.
Then, for any $r,s\ge1$ with $r+s=k\ge3$, we have $G^{\rm reg}_{r,s}(z)+G^{\rm reg}_{s,r}(z)+G_{k}(z) = \sum_{h+p=k} \left(\binom{p-1}{r-1}+\binom{p-1}{s-1}\right) G^{\rm reg}_{h,p}(z)$.
Moreover, from \eqref{eq:des}, it follows that $G^{\rm reg}_{r,s}(z)=G_{r,s}(z)$ for $r,s\ge2$.
Using this, we let
\[G_{r,s}^\frac12(q):=G_{r,s}^{\rm reg}(q)+\frac12 G_{r+s}(q)\]
for $r,s\ge1$ with $r+s\ge3$.

To understand the cuspidal relations, we observe two facts:
The set $\{G_{r,s}^\frac12(z) \mid r+s=k, r,s:{\rm odd}\}$ forms a basis of the $\C$-vector space spanned by $G_k(z)$ and $G_{r,s}^{\rm reg}(z)$ of weight $k$:
Every cusp form can be uniquely written as a $\C$-linear combination of this basis.
Giving an explicit formula for such a combination, we can recover the cuspidal relations for $\zeta^{\frac12}(odd,odd)$'s.

\begin{theorem}\label{thm:Hecke}{\rm \cite[Theorem 1]{Tasaka20}}
For a normalized Hecke eigenform $f\in S_k(\Gamma_1)$, let $a_{r,s}^f$ be as in Theorem \ref{thm:Refinement}.
Then we have
\begin{equation*}
\sum_{\substack{r+s=k\\r,s\ge1:{\rm odd}}} a_{r,s}^f G_{r,s}^\frac12(z)  = \frac{(2\pi i)^k}{4(k-2)!} \Lambda(f;1)f(z) .
\end{equation*}
\end{theorem}

\begin{example}
After multiplication by a constant (recall that the ratio $\Lambda(\Delta;s)/\Lambda(\Delta;s')$ of critical values is rational if $s\equiv s'\bmod2$), we get
\begin{equation}\label{eq:example}
\begin{aligned}
\frac{(2\pi i)^{12}}{680} \Delta(z) &= 22680 G_{9,3}^\frac12 (z)-35364 G_{7,5}^\frac12 (z) -29145 G_{5,7}^\frac12 (z) \\
&+13006 G_{3,9}^\frac12 (z) +22680 G_{1,11}^\frac12 (z),
\end{aligned}
\end{equation}
whose constant term yields the cuspidal relation \eqref{eq:cuspidal_rel}.
As another example, let $f=q+216q^2-3348q^3+\cdots $ be the unique normalized Hecke eigenform in $S_{16}(\Gamma_1)$.
Then we have
\begin{align*} 
\frac{(2\pi i)^{16}}{322560} f(z)&=1081080 G^{\frac12}_{1,15}(z)+842358 G^{\frac12}_{3,13}(z)-275295 G^{\frac12}_{5,11}(z)-1400182 G^{\frac12}_{7,9}(z)\\
   &-1360395 G^{\frac12}_{9,7}(z)-351252 G^{\frac12}_{11,5}(z)+1081080 G^{\frac12}_{13,3}(z) .
   \end{align*}
\end{example}

The proof of Theorem \ref{thm:Hecke} is done by combining Theorem \ref{thm:GKZ}, Kohnen-Zagier's extra relations of critical values \cite{KohnenZagier84} and Popa's decomposition formulas \cite{Popa} for Hecke eigenforms in terms of two product of the Eisenstein series.
The reason why Theorem \ref{thm:Hecke} holds only for cuspidal Hecke eigenforms is that Popa's formula is so.
In contrast, for odd weight, since Ma's relation (Theorem \ref{thm:Ma}) holds in the formal double zeta space, we can obtain
\[ \sum_{\substack{r+s=k\\r,s:{\rm odd}}} b_{r,s}^f G^{\frac12}_{r,s+1}(z)= 0,\
\sum_{\substack{r+s=k-1\\r:{\rm odd}}} c_{r,s}^f G^{\frac12}_{r,s}(z)=0, \]
where, for a cusp form $f$ of weight $k$, the above coefficients $b_{r,s}^f$ and $c_{r,s}^f$ are given in Theorem \ref{thm:Refinement} .

\section{Generalizations of period polynomial relations for double zeta values}

\subsection{Period polynomial relations for shuffle regularized double zeta values}
For $a,b,c\in \Z_{\ge0}$, Hirose \cite{Hirose23} studied the shuffle regularized double zeta value
\begin{equation}\label{ef:def_J}
\begin{aligned}
 J(a;b,c) &:= I_{\rm dch}(0^a,1,0^b,1,0^c)\\
 &=(-1)^a\sum_{\substack{i+j=a\\i,j\ge0}} \binom{b+i}{i}\binom{c+j}{j} \zeta(b+i+1,c+j+1),
 \end{aligned}
 \end{equation}
 where we mean $0^a=\underbrace{0,\ldots,0}_a$ and $I_{\rm dch}(a_1,\ldots,a_k)$ is an iterated integral of $\wedge_{j=1}^k \omega_{a_j}(t_j)$ with $\omega_0(t)=dt/t, \ \omega_1(t)=dt/(1-t)$, along the straight line path ${\rm dch}$ from the tangential basepoints $0'$ to $1'$. 
Note that $J(0;r-1,s-1)=\zeta(r,s)$.

A key observation is that the modular relation \eqref{eq:GKZ12} is in this setting written as follows:
\[ 28J(0;2,8) + 150 J(0;4,6) + 168J(0;6,4)=\frac{5197}{691}\zeta(12).\] 
Analogous to this, $\Q$-linear relations, for example, among $J(even;0,even)$'s, among $J(odd;1,odd)$'s, among $J(even;even,0)$'s and so on, are studied.
Some of them are given as follows.

\begin{theorem}\label{thm:Hirose}{\rm \cite[Theorem 18; (1.5), (1.6), (1.7), (1.9)]{Hirose23}}
For $w\ge2$ even, we have
\begin{align*}
{\rm (i)} \ & \sum_{a+b=w} p_{a,b}X^{a}Y^{b}\in W_{w,\Gamma_A}^{-}\Longrightarrow \sum_{a+b=w} a!b!\ p_{a,b}J(a;b,0)\equiv 0, \\
{\rm (ii)} \ & \sum_{a+b=w} p_{a,b}X^{a}Y^{b}\in W_{w}^+ \Longrightarrow \sum_{a+b=w} a!b!\ p_{a,b}J(a;b,0)\equiv 0, \\
{\rm (iii)} \ & \sum_{a+b=w} p_{a,b}X^{a}Y^{b}\in W_{w}^+ \Longrightarrow \sum_{\substack{a+b=w\\a,b:{\rm even}}} a!b!\ q_{a,b}J(a;0,b)\equiv 0, \\
{\rm (iv)} \ & \sum_{a+b=w} p_{a,b}X^{a}Y^{b}\in W_{w}^- \Longrightarrow \sum_{\substack{a+b=w\\a,b:{\rm odd}}} a!b!\ q_{a,b}J(a;1,b)\equiv 0.
\end{align*}
Here the convention we used is as follows.
The summation variables $a,b$ are integers $\ge0$.
The number $q_{a,b}$ on the right are defined by $\sum_{a+b=w} q_{a,b}X^{a}Y^{b} := P(X+Y,X)$ for $P\in V_{w}$.
The congruence is modulo the single zeta value $\zeta(k)$.
The space $W_{w,\Gamma_A}^{-} $ is the $\Q$-vector space of odd period polynomials corresponding to modular forms of weight $w+2$ on the congruence subgroup $\Gamma_A = \Gamma(2)\sqcup U\Gamma(2)\sqcup U^2 \Gamma(2)$ of level 2:
\begin{align*}
W_{w,\Gamma_A}^{-}&:=\{ P\in V_{w}^-\mid P\big|({1+U+U^2})=P\big|({S+SU+SU^2})=0\}.
 \end{align*}
\end{theorem}

An important point to note here is that Hirose's original results provide the opposite implication $(\Leftarrow)$, replacing $I_{\rm dch}$ with the motivic iterated integrals $I^{\mathfrak{m}}$.
Namely, for example, in the motivic setting, the relations obtained from Theorem \ref{thm:Hirose} (i) generate all $\Q$-linear relations among $J(odd_{\ge1};odd_{\ge1},0)$'s of weight $k$ modulo $\Q\zeta(k)$.
One missing thing in his results is the explicit formula for the coefficient of $\zeta(k)$.
In view of this situation, we provide these coefficients in the following examples (which can be computed by combining known linear relations of multiple zeta values).

\begin{example}
Let us illustrate a few examples of relations in Theorem \ref{thm:Hirose} (i).
Bases of $W_{k-2,\Gamma_A}^{-}$ for $k=6,8,10$ are given by $ W_{4,\Gamma_A}^{-}=\Q XY(X^2 - Y^2), \ W_{6,\Gamma_A}^{-} = 0, \  W_{8,\Gamma_A}^{-} = \Q XY(X^6-2X^4Y^2+2X^2Y^4-Y^6)$.
Non-trivial corresponding relations are
\begin{align*} 
&J(1;3,0)-J(3;1,0) =\frac{11}{6} \zeta(6),\\
&7J(1;7,0)-2J(3;5,0)+2J(5;3,0)-7 J(7;1,0)=\frac{29}{2}\zeta(10).
\end{align*}
Moreover, substituting \eqref{ef:def_J} into the above relations, we obtain
\begin{align*}
\frac{11}{6} \zeta(6)=&\zeta(2,4)+2\zeta(3,3)+2\zeta(4,2),\\
\frac{29}{2}\zeta(10)=&7\zeta(2,8)+14\zeta(3,7)+19\zeta(4,6)+20\zeta(5,5)+17\zeta(6,4)+14\zeta(7,3)+14\zeta(8,2).
\end{align*}
\end{example}

\begin{example}\label{ex:Hirose}
Using \eqref{ef:def_J}, we can recast Theorem \ref{thm:Hirose} (ii) as follows.
For $P\in W_{k-2}^+$, let $a_{r,s}$ be as in Theorem \ref{thm:GKZ}.
Then we have
\[ \sum_{r+s=k} a_{r,s} \, \zeta(r,s) \equiv 0 \bmod{\Q\zeta(k)}.\]
For example, we have
\begin{equation*}
\begin{aligned} 
\frac{6248}{691}\zeta(12)&=14{\zeta}(3,9)+42{\zeta}(4,8)+75{\zeta}(5,7)+95{\zeta}(6,6)+84{\zeta}(7,5)+42{\zeta}(8,4),\\
\frac{185656}{3617}\zeta(16)&=66 {\zeta}(3,13)+198 {\zeta}(4,12)+375 {\zeta}(5,11)+555 {\zeta}(6,10)+686
   {\zeta}(7,9)\\
 &+728 {\zeta}(8,8)+675 {\zeta}(9,7)+555 {\zeta}(10,6)+396
   {\zeta}(11,5)+198 {\zeta}(12,4).
\end{aligned}
\end{equation*}
\end{example}

\begin{example}
Taking $P=X^2Y^2(X^2-Y^2)^3\in W_{10}^+$ in Theorem \ref{thm:Hirose} (iii), we get
\[ 14J(2;0,8) + 75 J(4;0,6) + 84J(6;0,4) = \frac{59246}{691} \zeta(12).\]
Also, the special case $P=XY (X^2-Y^2)^2 (4X^4-17 X^2Y^2+4Y^4)\in W_{10}^-$ in Theorem \ref{thm:Hirose} (iv) gives
\[ 48 J(1;1,9) + 119 J(3;1,7) +10J(5;1,5) -144J(7;1,3)=640\zeta(13).\]
\end{example}

\begin{problem}
Find explicit formulas for the coefficients of $\zeta(k)$ in Theorem \ref{thm:Hirose}.
\end{problem}

\begin{problem}
Are there corresponding relations to Theorem \ref{thm:Hirose} in the formal double zeta space?
\end{problem}

\subsection{Period polynomial relations for double zeta values of level 2}
For positive integers $r\ge1,s\ge2$, Bachmann \cite{Bachmann20} studied the double zeta value
\[ \hat{\zeta}(r,s) := \sum_{0<m<n}\frac{1}{(m+n)^rn^s},\]
the special case of Apostol-Vu double zeta values or Witten zeta functions for $\mathfrak{so}(5)$.
Note that the value $\hat{\zeta}(r,s)$ is written in terms of double zeta values of level 2.
For example, one has $\hat{\zeta}(r,s)=2^{s-1}\big(Li\tbinom{1,-1}{r,s}+\zeta(r,s)\big) - \zeta(r,s)-\zeta(r+s)$ (see \cite[(4.6)]{Bachmann20}), where we set
\[ Li\tbinom{z_1,z_2}{k_1,k_2} := \sum_{0<m_1<m_2}\frac{z_1^{m_1}z_2^{m_2}}{m_1^{k_1}m_2^{k_2}}.\]

Using a $q$-analogue of $ \hat{\zeta}(r,s)$ which can be viewed as a holomorphic function on the complex upper half-plane, 
Bachmann showed a similar result to Theorem \ref{thm:Hecke}, i.e.,  explicit formulas for Hecke eigenforms on $\Gamma_1$ \cite[Theorem 1.1]{Bachmann20} (in contrast to Theorem \ref{thm:Hecke}, his formula does not give an expression in terms of a basis).
As a corollary, one can obtain $\Q$-linear relations of $\hat{\zeta}(r,s)$'s from even period polynomials.

\begin{theorem}\label{thm:Bachmann}
For $P\in W_{k-2}^+$, let $a_{r,s}$ be as in Theorem \ref{thm:GKZ}.
Then we have
\[ \sum_{r+s=k} a_{r,s} \, \hat{\zeta}(r,s) \equiv 0 \bmod{\Q\zeta(k)}.\]
\end{theorem}

Note that the resulting formula in Theorem \ref{thm:Bachmann} is similar to, but essentially different from Example \ref{ex:Hirose} (Theorem \ref{thm:Hirose} (ii)).

\begin{example}
Taking $P=X^{k-2}-Y^{k-2}$, we get
\[ \sum_{\substack{r+s=k\\ s\ge2}}\hat{\zeta}(r,s)=\frac{1}{2^{k-1}} \zeta(k).\]
The corresponding relations to $X^2Y^2(X^2-Y^2)^3\in W_{10}^{+,0}$ and $X^2Y^2(X^2-Y^2)^3(2X^4-X^2Y^2+2Y^4)\in W_{14}^{+,0}$ are
\begin{equation}\label{eq:zetahat}
\begin{aligned} 
\frac{1639}{2^8.691}\zeta(12)&=14\hat{\zeta}(3,9)+42\hat{\zeta}(4,8)+75\hat{\zeta}(5,7)+95\hat{\zeta}(6,6)+84\hat{\zeta}(7,5)+42\hat{\zeta}(8,4),\\
\frac{58703}{2^{12}.3617}\zeta(16)&=66 \hat{\zeta}(3,13)+198 \hat{\zeta}(4,12)+375 \hat{\zeta}(5,11)+555 \hat{\zeta}(6,10)+686
   \hat{\zeta}(7,9)\\
 &+728 \hat{\zeta}(8,8)+675 \hat{\zeta}(9,7)+555 \hat{\zeta}(10,6)+396
   \hat{\zeta}(11,5)+198 \hat{\zeta}(12,4).
\end{aligned}
\end{equation}
For the explicit formulas for the coefficients of $\zeta(k)$ in the above relations, see \cite{Bachmann20}.
\end{example}

If $P\in W_{k-2}^{+,0}$, then we have $a_{1,k-1}=a_{2,k-2}=a_{k-3,3}=a_{k-2,2}=a_{k-1,1}=0$.
Therefore, Theorem \ref{thm:Bachmann} for the case $P\in W_{k-2}^{+,0}$ gives a $\Q$-linear relation among $\zeta(k)$ and $\hat{\zeta}(r,s) \ (r\ge3,s\ge4)$ of weight $k$.
On the other hand, we can observe that every $\Q$-linear relation among them does not come from Theorem \ref{thm:Bachmann}, so the situation is different from the cases $\zeta(odd_{\ge1},odd_{\ge3})$ and $\zeta(odd_{\ge3},even_{\ge4})$.
Here is the list of the numerical dimension of the $\Q$-vector space spanned by $\zeta(k)$ and $\hat{\zeta}(r,s) \ (r\ge3,s\ge4)$ of weight $k$.

\begin{center}
\begin{tabular}{c|cccccc}
$k$ & 8 & 10 & 12 & 14 & 16 & 18\\ \hline
$\sharp$ of generators & 3& 5 & 7 & 9 & 11 &13\\ \hline
dimension & 3 & 4 & 5 & 6 & 7 & 8\\ \hline
\end{tabular}
\end{center}

\subsection{Conjectural period polynomial relations for double $\widetilde{T}$-values}

For positive integers $k_1,\ldots,k_d$, Kaneko and Tsumura \cite{KanekoTsumura} introduced the multiple $\widetilde{T}$-value $ \widetilde{T} (k_1,\ldots,k_d)$ defined by
\[ \widetilde{T} (k_1,\ldots,k_d):=2^d \sum_{\substack{0<m_1<\cdots<m_d\\m_j\equiv j\bmod 2\\j=1,2,\ldots,d}} \frac{(-1)^{\frac{m_d-d}{2}}}{m_1^{k_1}\cdots m_d^{k_d}}.\]
Similarly to multiple zeta values, the multiple $\widetilde{T}$-value has an iterated integral expression with the integrands $2dt/(1+t^2)$ and $dt/t$ (cf.~\cite[Proposition 2.1]{KanekoTsumura}).
Thus, the $\Q$-vector space spanned by all multiple $\widetilde{T}$-values forms a $\Q$-algebra with the product given by the shuffle product.
The integral representation also leads to an expression in terms of multiple zeta values of level 4.
For example, we have
\[ \widetilde{T}(r,s)=-Li\tbinom{1,i}{r,s}+Li\tbinom{-1,-i}{r,s} + Li\tbinom{-1,i}{r,s} - Li\tbinom{1,-i}{r,s}.\]

Before going to a modular relation for double $\widetilde{T}$-values found by 
Kaneko and Tsumura, let us observe the numerical dimension of the $\Q$-vector space spanned by all double $\widetilde{T}$-values of weight $k$.
Set
\[  \mathcal{D}\widetilde{\mathcal{T}}_k := \left\langle \widetilde{T}(k-s,s)\ \middle|\ 1\le s\le k-1\right\rangle_\Q.\]
\begin{center}
\begin{tabular}{c|cccccccccccccc}
$k$ &2& 3&4&5&6&7&8&9&10&11&12&13&14\\ \hline
$\sharp$ of generators &1& 2&3&4&5&6&7&8&9&10&11&12&13\\\hline
$\dim_\Q \mathcal{D}\widetilde{\mathcal{T}}_k $ &1&2&3&4&4&6&6&8&7&10&9 &12 &10
\end{tabular}
\end{center}
From this and further data, for $k\ge2$, we may expect that
\begin{equation}\label{eq:dim_conj_tilde{T}}
\dim_\Q \mathcal{D}\widetilde{\mathcal{T}}_k \stackrel{?}{=} \begin{cases}k-1-\left[ \frac{k-2}{4}\right] & k:{\rm even},\\ k-1 &k:{\rm odd}. \end{cases}.
\end{equation}
Note that the term $k-1$ is the number of generators of the space $\mathcal{D}\widetilde{\mathcal{T}}_k$, and that for $k\ge2$ even, we have
\[ \left[ \frac{k-2}{4}\right]=\dim S_k(\Gamma_0(4))-\dim S_k(\Gamma_0(2)).\]
This implies that the values $\widetilde{T}(r,s)$ of weight $k$ will satisfy $\dim S_k(\Gamma_0(4))-\dim S_k(\Gamma_0(2))$ relations over $\Q$.

To describe such relations, for each even integer $w\ge2$ and $1\le j\le w/2$, define the polynomial $\widetilde{S}_{w,j}(X,Y)\in V_w^+$ by
\begin{align*}
\widetilde{S}_{w,j}(X,Y)&:=\frac{4^{w-2j+1}}{w+2-2j} B_{w,w+2-2j}(Y/4,X)-\frac{1}{2j} B_{w,2j}(X,Y) \\
& \qquad -\frac{(w+2) B_{2j}B_{w+2-2j}}{2j (w+2-2j)B_{w+2}} \left(\frac{1-2^{-2j}}{1-2^{-w-2}}\frac{X^w}{4}- \frac{1-2^{-w-2+2j}}{1-2^{-w-2}}\frac{Y^w}{4^{2j}}\right),
\end{align*}
where we set 
\[B_{w,n}(X,Y):=\sum_{\substack{0\le j\le n\\ j:{\rm even}}} \binom{n}{j} X^{n-j}Y^{w-n+j}.\]
Note that $\widetilde{S}_{w,j}(X,Y)$ is obtained from $r^+(R_{\Gamma_0(4),w,2j-1}) \ (1\le j \le w/2)$ which appear in the context of ``rational periods of cusp forms on $\Gamma_0(N)$" studied by Fukuhara and Yang \cite[Theorem 1.1]{FukuharaYang}.
Let $W_{w,4}^+$ denote the $\Q$-vector subspace of $V_w^+$ spanned by $\widetilde{S}_{w,j}(X,Y) \ (1\le j\le w/2)$.
For $w\ge2$ even, it was shown in \cite[Corollary 1.9]{FukuharaYang2} that
\[\dim_\Q W_{w,4}^+ = \dim_\C S_{w+2}(\Gamma_0(4))=\frac{w}{2}-1.\]

\begin{conjecture}\label{conj:KT-conjecture}{\rm \cite[Conjecture 2.12, 1)]{KanekoTsumura}}
Let $k\ge6$ even.
For $P\in W_{k-2,4}^+$, let $Q(X,Y)=P(X+Y,-2X+2Y)$ and set $Q^+=\frac12 Q\big|(1+\delta)\in V_{k-2}^+$.
Then we have that \[\dim_\Q \langle Q^+ \mid P\in W_{k-2,4}^+\rangle_\Q \stackrel{?}{=} \dim_\C S_k(\Gamma_0(4))- \dim_\C S_k(\Gamma_0(2)).\]
Moreover, define $d_{r,s}\in \Q$ by 
\[ \sum_{r+s=k} \binom{k-2}{r-1}d_{r,s}X^{r-1}Y^{s-1} :=Q^+(X+Y,X) .\]
Then we have that
\[ \sum_{r+s=k} d_{r,s}\widetilde{T}(r,s) \stackrel{?}{=} 0.\]
\end{conjecture}

Conjecture 2.12 in \cite{KanekoTsumura} includes other kinds of modular relations of level 2, but we only focus on the above case.

\begin{example}
For the case
\[ \widetilde{S}_{4,1}(X,Y)=- \frac12 X^4 +\frac12 X^2Y^2 -\frac{1}{32}Y^4 \in W_{4,4}^+,\]
we have $Q^+(X,Y)=X^4-10X^2Y^2+Y^4$.
This leads to the conjectural relation of the form
\[ 24 \widetilde{T}(1, 5) + 12 \widetilde{T}(2, 4) + 2 \widetilde{T}(3, 3) - 3 \widetilde{T}(4, 2) - 3 \widetilde{T}(5, 1)=0.\]
\end{example}

\begin{remark}
To prove the upper bound of the dimension conjecture \eqref{eq:dim_conj_tilde{T}}, Brown's method using the theory of motivic multiple zeta values (cf.~\cite{Brown12}) would be applicable.
Let us sketch an outline without going into details.
Denote by $\widetilde{T}^{\mathfrak{m}}(r,s)$ the motivic version of double $\widetilde{T}$-values.
An advantage of the use of the motivic version is the fact that for $\xi=\sum a_{k_1,k_2}\widetilde{T}^\mathfrak{m}(k_1,k_2)$, the $\Q$-linear combination $\xi$ is a constant multiple of the single zeta value $\zeta^{\mathfrak{m}}(k)$ (or rather, the $k$th power of the motivic $2\pi i$) if and only if $\overline{\Delta}(\xi) = 0$, where $\overline{\Delta}(x):=\Delta (x) - 1\otimes x -x\otimes 1$ is the reduced Goncharov coproduct $\Delta$.
Here the Goncharov coproduct can be computed explicitly using formal iterated integrals;
for our case, we have
\begin{align*}
\overline{\Delta}\big(\widetilde{T}^{\mathfrak{m}} & (k_1,k_2)\big)
= T^{\mathfrak{a}}(k_1)\otimes \zeta_R^{\mathfrak{m}}(k_2)\\
&+\sum_{l_1+l_2=k_1+k_2} \left\{(-1)^{k_1}\binom{l_1-1}{k_1-1}T^{\mathfrak{a}}(l_1)\otimes \zeta_R^{\mathfrak{m}} (l_2)+(-1)^{l_1-k_2}\binom{l_1-1}{k_2-1} \widetilde{T}^{\mathfrak{a}} (l_1)\otimes \widetilde{T}^{\mathfrak{m}}(l_2) \right\},
\end{align*}
where $\zeta_R^\mathfrak{m}$ and $T^\mathfrak{m}$ are motivic lifts of the values $\zeta_R(k)=\frac{1}{4^{k-1}}(2^{k-1}-1)\zeta(k)$ and $ T(k)=2(1-2^{-k})\zeta(k)$.
Here by $T^\mathfrak{a}$ and $\widetilde{T}^\mathfrak{a}$ we denote $T^\mathfrak{m}$ and $\widetilde{T}^\mathfrak{m}$ modulo the motivic $2\pi i$, respectively.
We have $T^\mathfrak{a}(2r)=0$ and $\widetilde{T}^\mathfrak{a}(2r-1)=0$ for any $r\ge1$.
Since the terms $T^{\mathfrak{a}}(l_1)\otimes \zeta_R^{\mathfrak{m}} (l_2)$ $(l_1:{\rm odd})$ and $\widetilde{T}^{\mathfrak{a}} (l_1)\otimes \widetilde{T}^{\mathfrak{m}}(l_2)$ $(l_1:{\rm even})$ are linearly independent over $\Q$, every left annihilator $(a_{k_1,k_2})$ of the $(k-1)\times (k+k/2-4)$ matrix $C_k := [{C}_k':C_k'']$ gives rise to the relation $\sum a_{k_1,k_2}\widetilde{T}^\mathfrak{m}(k_1,k_2)\equiv 0$, where the entries of the matrices $C_k'$ and $C_k''$ are obtained from the coefficients of $T^{\mathfrak{a}}(l_1)\otimes \zeta_R^{\mathfrak{m}} (l_2)$ and $\widetilde{T}^{\mathfrak{a}} (l_1)\otimes \widetilde{T}^{\mathfrak{m}}(l_2)$ in $\overline{\Delta}\big(\widetilde{T}^{\mathfrak{m}}(k_1,k_2)\big)$, respectively.
More precisely, $C_k' $ is the $(k-1)\times (k-3)$ matrix given by
\[ C_k' := \left( \delta_{l_1,k_1}+(-1)^{k_1}\binom{l_1-1}{k_1-1} \right)_{\begin{subarray}{c} k_1+k_2=k, k_i\ge1\\ l_1+l_2=k, l_i\ge2 \end{subarray}},\]
and 
$C_k''$ is the $(k-1)\times(k/2-1)$ matrix defined by
\[C_k'': =  \left( (-1)^{l_1-k_2}\binom{l_1-1}{k_2-1} \right)_{\begin{subarray}{c} k_1+k_2=k, k_i\ge1\\ l_1+l_2=k, l_i\ge2:{\rm even} \end{subarray}}\]
where the rows and columns are indexed by $(k_1,k_2)$ and $(l_1,l_2)$.
For example, we have
\[{C}_4'=\left(
\begin{array}{c}
 -1 \\
 2 \\
 0 \\
\end{array}
\right),\ {C}_6'=
\left(
\begin{array}{ccc}
 -1 & -1 & -1 \\
 2 & 2 & 3 \\
 0 & 0 & -3 \\
 0 & 0 & 2 \\
 0 & 0 & 0 \\
\end{array}
\right),\ {C}_8'=
\left(
\begin{array}{ccccc}
 -1 & -1 & -1 & -1 & -1 \\
 2 & 2 & 3 & 4 & 5 \\
 0 & 0 & -3 & -6 & -10 \\
 0 & 0 & 2 & 4 & 10 \\
 0 & 0 & 0 & 0 & -5 \\
 0 & 0 & 0 & 0 & 2 \\
 0 & 0 & 0 & 0 & 0 \\
\end{array}
\right),
\]
and 
\[ {C}_4''=\left(
\begin{array}{c}
 0 \\
 1 \\
 -1 \\
\end{array}
\right),\ {C}_6''=\left(
\begin{array}{cc}
 0 & 0 \\
 0 & 1 \\
 0 & -3 \\
 1 & 3 \\
 -1 & -1 \\
\end{array}
\right),\ 
{C}_8''= \left(
\begin{array}{ccc}
 0 & 0 & 0 \\
 0 & 0 & 1 \\
 0 & 0 & -5 \\
 0 & 1 & 10 \\
 0 & -3 & -10 \\
 1 & 3 & 5 \\
 -1 & -1 & -1 \\
\end{array}
\right).\]
For $k\ge4$ even, we observed the dimension of the $\Q$-vector space $\ker C_k$ of left annihilators of the matrix $C_k$ and it is expected to be
\[\dim_\Q \ker C_k \stackrel{?}{=} \left[ \frac{k-2}{4} \right]+1= \dim S_k(\Gamma_0(4))-\dim S_k(\Gamma_0(2)) +1.\]
Note that, according to \eqref{eq:dim_conj_tilde{T}}, the above expectation would also imply $\zeta(k) \in \mathcal{D}\widetilde{\mathcal{T}}_k$, but this is already known from the weighted sum formula
\[ \sum_{j=0}^{k-2}2^{k-j-2}\widetilde{T}(j+1,k-1-j) + \widetilde{T}(k-1,1)=(k-1)T(k).\]
See the end of \S2.5 in \cite{KanekoTsumura}.
With tedious calculation, it might be possible to prove that for $d_{r,s}$ in Conjecture \ref{conj:KT-conjecture}, its vector $(d_{r,s})$ becomes a left annihilator of $C_k$.
\end{remark}

\begin{problem}
Prove \eqref{eq:dim_conj_tilde{T}}.
\end{problem}

\begin{problem}
Mimicking the story of cuspidal relations, we consider double Eisenstein series
\[ \widetilde{H}_{r,s}(z) := 2^2 \sum_{\substack{0\prec 4\ell_1z+m_1\prec 4\ell_2z+m_2 \\m_i\equiv i\bmod 2,\ \ell_i,m_i\in \Z}} \frac{(-1)^{\frac{m_2-2}{2}}}{(4\ell_1z +m_1)^{r} (4\ell_2 z+m_2)^{s}}.\]
Its constant term of the Fourier expansion coincides with $\widetilde{T}(r,s)$.
We might expect that there are regularizations of $\widetilde{H}_{r,s}(z)$ for $r,s\ge1$ such that the linear combination $\sum_{r+s=k} d_{r,s}\widetilde{H}_{r,s}(z)$, where $d_{r,s}$ is defined in Conjecture \ref{conj:KT-conjecture}, is a cusp form on $\Gamma_0(4)$.
If so, characterize the subspace spanned by such cusp forms.
Note that we may apply some results of Yuan and Zhao \cite{YuanZhao16} to the regularization of $\widetilde{H}_{r,s}(z)$  (see also \cite{KanekoTasaka13} for $N=2$).
\end{problem}

\subsection{Period polynomial relations for colored double zeta values}

For $N\in \N$, let $\mu_N$ be the set of all $N$th roots of unity.
For $a_1,\ldots,a_d\in \Z/N\Z$ and $k_1,\ldots,k_d\in \N$ with $(k_d,a_d)\neq (1,1)$, we define the colored multiple zeta value $\zeta\tbinom{a_1,\ldots,a_d}{k_1,\ldots,k_d}$ relative to $\mu_N $ by
\[
 \zeta\tbinom{a_1,\ldots,a_d}{k_1,\ldots,k_d}: = \sum_{0<m_1<\cdots<m_d}\prod_{j=1}^d\frac{\eta_N^{a_jm_j}}{m_j^{k_j}},
\]
 where we let $\eta_N=e^{\frac{2\pi i }{N}}$.

Recently, for each $N\in \N$, Hirose \cite{Hirose} obtained modular relations for colored double zeta values relative to $\mu_N$.
His result is about a correspondence between relations in the formal ``colored" double zeta space and period polynomials for a congruence subgroup due to Pa\c{s}ol and Popa \cite{PasolPopa}, which generalizes Theorem \ref{thm:GKZ}.

Firstly, let us define the formal colored double zeta space.
For $N\ge1$ and $k\ge2$, let $A(N)=\{(a,b)\in (\Z/N\Z)^2\mid (a,b,N)=1\}$ and define the $\Q$-vector space
\[ \mathcal{D}_{k,N}:=\langle Z_k^c,Z_{r,s}^{a,b},P_{r,s}^{a,b}\mid r+s=k, r,s\ge1, c\in \Z/N\Z, (a,b)\in A(N)\rangle_\Q \]
spanned by symbols $Z_k^c,Z_{r,s}^{a,b},P_{r,s}^{a,b}$ satisfying the regularized double shuffle relations
\[ P_{r,s}^{a,b}=Z_{r,s}^{a,b}+Z_{s,r}^{b,a}+Z_{k}^{a+b} = \sum_{h+p=k} \left(\binom{p-1}{r-1}Z_{h,p}^{b-a,a}+\binom{p-1}{s-1}Z_{h,p}^{a-b,b} \right)  \]
for $(a,b)\in A(N)$ and $r,s\ge1$ with $r+s=k$.
The above relations are well-defined, because, for $(a,b)\in A(N)$, the pairs $(b-a,a)$ and $(a-b,b)$ are also in $A(N)$.
In a similar manner to Proposition \ref{prop:dsr}, we can show that colored double zeta values satisfy the same relations above (cf.~\cite{ArakawaK04,Racinet02}).
Hence, there is the natural surjection $\mathcal{D}_{k,N}\rightarrow \mathcal{DZ}_{k,N}$, where $\mathcal{DZ}_{k,N}$ is the $\Q$-vector space spanned by all colored multiple zeta values relative to $\mu_N$ of weight $k$ and depth $\le2$.
For latter purpose, denote by $\mathcal{P}_{k,N}^{\rm ev}$ the subspace of $\mathcal{D}_{k,N}$ generated by $Z_k^c+(-1)^k Z_k^{-c} \ (c\in \Z/N\Z)$ and
\[ P_{r,s}^{a,b}+(-1)^r P_{r,s}^{-a,b}+(-1)^sP_{r,s}^{a,-b}+(-1)^{r+s}P_{r,s}^{-a,-b} \quad  (r+s=k, \ (a,b)\in A(N)). \]
Note that the image of the space $\mathcal{P}_{k,N}^{\rm ev}$ under the map $Z\mapsto \zeta$ becomes $\Q(2\pi i)^k$.

Next, recall period polynomials for $\Gamma_1(N)$ due to Pa\c{s}ol and Popa \cite{PasolPopa}.
For simplicity, we only treat the case $w\ge0$ even.
For a polynomial-valued function $F:\Gamma_1(N)\backslash {\rm SL}_2(\Z) \rightarrow V_w$, define the right action of $\gamma\in {\rm SL}_2(\Z)$ by $(F\big|_\gamma)(C):=F(C\gamma^{-1})\big|\gamma$ for each representative $C$ for the right cosets $\Gamma_1(N)\backslash {\rm SL}_2(\Z)$.
Extending this action to the group ring $\Z[{\rm SL}_2(\Z)]$, for $w\ge0$ even and $N\in \N$, we let
\[ \mathcal{W}_{w,N}:= \{F:\Gamma_1(N)\backslash {\rm SL}_2(\Z) \rightarrow V_w\mid F\big|_{1+S}=F\big|_{1+U+U^2}=0 \ \mbox{and}\  F(-C)=F(C),\ \forall C\}.\]
Denote by $C_{a,b}$ the corresponding class to $(a,b)\in A(N)$ via the bijection $\Gamma_1(N)\backslash {\rm SL}_2(\Z) \rightarrow A(N), \ \Gamma_1(N)(\begin{smallmatrix}\ast&\ast\\a&b\end{smallmatrix})\mapsto (a,b)$.
Under this notation, every function $F\in \mathcal{W}_{w,N}$ is identified with the polynomial vector $(F(C_{a,b}))_{(a,b)\in A(N)}\in V_w^{|A(N)|}$ such that
\[0=F(C_{a,b}) +F(C_{a,b}S^{-1})\big|S=F(C_{a,b}) + F(C_{a,b}U^{-1})\big|U+F(C_{a,b}U^{-2})\big|{U^2} \]
and $F(C_{-a,-b})=F(C_{a,b})$ for all $(a,b)\in A(N)$. 
We set
\[ \mathcal{W}_{w,N}^{\pm} = \{ F\in \mathcal{W}_{w,N}\mid F(C_{a,-b}) \big|\delta=\pm F(C_{a,b}) \ \mbox{for each $(a,b)\in A(N)$}\}.\]
It follows that $\mathcal{W}_{w,1}^\pm$ is canonically isomorphic to $W_w^\pm$.
From the Eichler-Shimura theory, one can construct injections $r^\pm : S_{w+2}(\Gamma_1(N))\rightarrow \mathcal{W}_{w,N}^\pm\otimes_\Q\C$.

We now state modular relations for colored double zeta values relative to $\mu_N$ due to Hirose.

\begin{theorem}\label{thm:colored}
Let $k\ge4$ even.
For $F\in \mathcal{W}_{k-2,N}^+$ and $(a,b)\in A(N)$, define $e_{r,s}^{a,b}\in\Q$ by 
\begin{align*}
\sum_{r+s=k}\binom{k-2}{r-1}e_{r,s}^{a,b} X^{r-1}Y^{s-1}:=F(C_{a,-a+b})(X-Y,X)
\end{align*}
and set 
\[e_{r,s}^{a,b,{\rm ev}}=\frac12 (e_{r,s}^{a,b} + (-1)^r e_{r,s}^{a,b}), \quad e_{r,s}^{a,b,{\rm od}}=\frac12 (e_{r,s}^{a,b} - (-1)^r e_{r,s}^{a,b}).\]
Then we have 
\[e_{r,s}^{a,b,{\rm ev}}=(-1)^{r}e_{r,s}^{-a,b,{\rm ev}}=(-1)^{s}e_{r,s}^{a,-b,{\rm ev}} = e_{r,s}^{b,a,{\rm ev}} ,\quad e_{r,s}^{a,b,{\rm od}}=(-1)^{r+1}e_{r,s}^{-a,b,{\rm od}}=(-1)^{s+1}e_{r,s}^{a,-b,{\rm od}} \] 
and 
\[3 \sum_{\substack{r+s=k\\ (a,b)\in A(N)}} e_{r,s}^{a,b,{\rm od}} Z_{r,s}^{a,b} = - \sum_{\substack{r+s=k\\ (a,b)\in A(N)}} e_{r,s}^{a,b,{\rm ev}} Z_{r,s}^{a,b} - \sum_{\substack{r+s=k\\ (a,b)\in A(N)}}  e_{r,s}^{a,b} Z_{r+s}^{a+b} \in \mathcal{P}_{k,N}^{\rm ev}.\]
\end{theorem}

Hirose's statement in \cite[Theorem 3]{Hirose} is much stronger than the above, including the case $k$ odd.

\begin{example}
Let us illustrate the case $N=2$.
Since $A(2)=\{(0,1),(1,0),(1,1)\}$, letting $F_{a,b}=F(C_{a,b})$, we identify $F\in \mathcal{W}_{w,2}$ with $(F_{0,1},F_{1,0},F_{1,1})\in V_w^3$ such that
\[ F_{0,1}+F_{1,0}\big|S=0, \ F_{1,1}\big|(1+S)=0, \ F_{0,1}+F_{1,1}\big|U + F_{1,0}\big|U^2=0.\] 
For example, we see that $F=(x^6, -y^6, x^6 - y^6)\in \mathcal{W}_{6,2}^+$ (which lies in the coboundary part) and this gives 
\begin{align*}
&3 \big(Z^{0, 1}_ {1, 7}+ Z^{0, 1}_{3, 5} +Z^{0, 1}_{5, 3} +Z^{0, 1}_{7, 1}+Z^{1, 0}_{1, 7}+ Z^{1, 0}_ {3, 5} + Z^{1, 0}_{5, 3} -Z^{1, 1}_{7, 1}\big)\\
&=Z^{0}_{8} - Z^{1}_{8} +  Z^{0, 1}_{2, 6} + Z^{0, 1}_{4, 4} + Z^{0, 1}_{6, 2} + Z^{1, 0}_{2, 6} + Z^{1, 0}_{4, 4} + Z^{1, 0}_{2, 6}\ \longrightarrow \frac{3}{2}\zeta(8)  \quad \big(\mbox{under}\ {Z\longmapsto\zeta}\big).
 \end{align*}
Taking $F=\left(15 x^6-30 x^4 y^2+15 x^2 y^4,-15 x^4 y^2+30 x^2 y^4-15 y^6,45 x^4 y^2-45 x^2 y^4\right)\in \mathcal{W}_{6,2}^+$, we get
\begin{align*}
&3\big(15Z^{0, 1}_{1, 7}+13Z^{0, 1}_{3, 5} + 4 Z^{0, 1}_{5, 3} + 3 Z^{1, 0}_{3, 5} + 15 Z^{1, 0}_{5, 3} - Z^{1,1}_{3,5} -4Z^{1,1}_{5,3}\big)\\
&=2Z^{0}_{8} - 2 Z^{1}_{8} + 15 Z^{0,1}_{2,6} + 9 Z^{0,1}_{4,4} + 9 Z^{1,0}_{4,4} + 15 Z^{1,0}_{6,2} - 3 Z^{1,1}_{4,4}\ \longrightarrow \frac{1395}{256}\zeta(8)  \quad \big(\mbox{under}\ {Z\longmapsto\zeta}\big).
\end{align*}
\end{example}

\begin{problem}
As a corollary of Theorem \ref{thm:colored}, for each $F\in \mathcal{W}_{k-2,N}^+$ we obtain the relation of the form
\[\sum_{\substack{r+s=k\\ (a,b)\in A(N)}} e_{r,s}^{a,b,{\rm od}} \zeta^\shuffle\tbinom{a,b}{r,s} = c_F \cdot (2\pi i)^k,\]
where $\zeta^\shuffle$ denotes the shuffle regularized colored multiple zeta values. 
Similarly to Theorem \ref{thm:Refinement}, is there a simple formula for the constant $c_F$ when $F$ is cuspidal?
\end{problem}

\begin{problem}
Is Conjecture \ref{conj:KT-conjecture} derived from Theorem \ref{thm:colored}?
\end{problem}

\subsection*{Acknowledgments}
I would like to thank Professor Maki Nakasuji for giving me the opportunity to talk at the RIMS conference on ``Zeta functions and their representations."
I am also grateful to Professor Masanobu Kaneko for his corrections and comments.
This work is partially supported by
JSPS KAKENHI Grant Number 20K14294, the Nitto Foundation, and the Research Institute for Mathematical Sciences,
an International Joint Usage/Research Center located in Kyoto University.


\vspace{2ex}
\noindent
Koji Tasaka\\
School of Information Science \& Technology, \\
Aichi Prefectural University, \\
Nagakute 480-1198, JAPAN\\
E-mail address: tasaka@ist.aichi-pu.ac.jp


\begin{thebibliography}{10}

\bibitem{ArakawaK04} T.~Arakawa, M.~Kaneko, {\itshape On multiple L-values}, J.~Math.~Soc.~Japan, {\bf 56} (2004), 967--991.

\bibitem{Bachmann19} H.~Bachmann, {\itshape The algebra of bi-brackets and regularized multiple Eisenstein series}, J.~Number Theory, {\bf 200}  (2019), 260--294.


\bibitem{Bachmann20} H.~Bachmann, {\itshape Modular forms and $q$-analogues of modified double zeta values}, Abh. Math.~Seminar Uni.~Hamburg, {\bf 90} (2020), 201--213.

\bibitem{BachmannTasaka18} H.~Bachmann, K.~Tasaka, {\itshape The double shuffle relations for multiple Eisenstein series}, Nagoya Math.~J., {\bf 230} (2018), 180--212.

\bibitem{BaumardSchneps13} S.~Baumard, L.~Schneps, {\itshape Period polynomial relations between double zeta values}, Ramanujan J., {\bf 32}(1) (2013), 83--100.

\bibitem{BordhurstKreimer97} D.~Broadhurst, D.~Kreimer, {\itshape Association of multiple zeta values with positive knots via Feynman diagrams up to 9 loops}, Phys.~Lett.~B, {\bf 393} (1997), no. 3-4, 403--412.

\bibitem{Brown12} F.~Brown, {\itshape Mixed Tate motives over $\Z$}, Ann.~of Math., {\bf 175} (2012), no.~2, 949--976.

\bibitem{Brown14} F.~Brown, {\itshape Multiple modular values and the relative completion of the fundamental group of $M_{1,1}$}, preprint (arXiv:1408.5167v4).

\bibitem{Brown17} F.~Brown, {\itshape Zeta elements in depth 3 and the fundamental Lie algebra of a punctured elliptic curve}, Forum Math.~Sigma, {\bf 5} (2017), 1--56.

\bibitem{Brown21} F.~Brown, {\itshape Depth-graded motivic multiple zeta values}, Compositio Math., {\bf 157}(3) (2021), 529--572.

\bibitem{CarrGanglSchneps15} S.~Carr, H.~Gangl, L.~Schneps, {\itshape On the Broadhurst-Kreimer generating series for multiple zeta values}, in Feynman Amplitudes, Periods and Motives (K. Ebrahimi-Fard et. al, eds.), AMS {\bf 648} (2015), 57--77.

\bibitem{DS} F.~Diamond, J.~Shurman, {\itshape A First Course in Modular Forms}, Graduate Texts in Mathematics, {\bf 228}, Springer New York, 2005.

\bibitem{DMNW} C.~Dietze, C.~Manai, C.~N\"{o}bel, F.~Wagner, {\itshape Totally odd depth-graded multiple zeta values and period polynomials}, preprint, arXiv:1708.07210v1.

\bibitem{EnriquezLochak16} B.~Enriquez, P.~Lochak, {\itshape Homology of depth-graded motivic Lie algebras and koszulity}, J.~ Th\'{e}or.~Nombres Bordeaux {\bf 28} (2016), no.~3, 829--850.

\bibitem{FukuharaYang} S.~Fukuhara, Y.~Yang, {\itshape Period polynomials and explicit formulas for Hecke operators on $\Gamma_0(2)$}, Math.~Proc.~Cambridge Philos.~Soc., {\bf 146} (2009), no. 2, 321--350.

\bibitem{FukuharaYang2} S.~Fukuhara, Y.~Yang, {\itshape A basis for $S_k(\Gamma_0(4))$ and representations of integers as sums of squares}, Ramanujan J., {\bf 28} (2012), no. 1, 25--43.

\bibitem{GanglKanekoZagier06} H.~Gangl, M.~Kaneko, D.~Zagier, {\itshape Double zeta values and modular forms}, Automorphic forms and Zeta functions, In:Proceedings of the conference in memory of Tsuneo Arakawa, World Scientific, (2006), 71--106. 

\bibitem{Goncharov01} A.B.~Goncharov, {\itshape The dihedral Lie algebras and Galois symmetries of $\pi^{(l)}_1 ({\mathbb P}^1-(\{0,\infty\}\cup \mu_N))$}, Duke Math.\ J., {\bf 110}(3) (2001), 397--487.

\bibitem{Hirose23} M.~Hirose, {\itshape Modular phenomena for regularized double zeta values}, to appear in Isr.~J.~Math..

\bibitem{Hirose} M.~Hirose, {\itshape Colored double zeta values and modular forms of general level}, preprint (arXiv:2205.08507, v1).

\bibitem{IharaKanekoZagier06} K.~Ihara, M.~Kaneko, D.~Zagier,  {\itshape Derivation and double shuffle relations for multiple zeta values}, Compositio Math. {\bf 142} (2006), 307--338.

\bibitem{Kaneko07} M.~Kaneko, {\itshape Double zeta values and modular forms}. In: Kim, H.K., Taguchi, Y.~(eds.) Proceedings of the Japan--Korea joint seminar on Number Theory, Kuju, Japan (2004).


\bibitem{KanekoTasaka13} M.~Kaneko, K.~Tasaka, {\itshape Double zeta values, double Eisenstein series, and modular forms of level 2}, Math.~Ann.~{\bf 357} (2013), no.~3, 1091--1118.

\bibitem{KanekoTsumura} M.~Kaneko, H.~Tsumura, {\itshape Multiple L-values of level four, poly-Euler numbers, and related zeta functions}, to appear in Tohoku Math. J..

\bibitem{KohnenZagier84} W.~Kohnen, D.~Zagier, {\itshape Modular forms with rational periods}, Modular forms (Durham, 1983), Ellis Horwood (1984), 197--249.

\bibitem{Lang76} S.~Lang, {\itshape Introduction to modular forms}, Grundlehren der mathematischen Wissenschaften,
No. 222. Springer-Verlag, Berlin-New York, 1976.

\bibitem{Li19} J.~Li, {\itshape The depth structure of motivic multiple zeta values}, Math.~Ann., {\bf 374} (2019), 179--209.

\bibitem{LiLiu} J.~Li, F.~Liu, {\itshape Motivic double zeta values of odd weight}, Manuscripta Math., {\bf 166} (2021), 19--36.

\bibitem{Manin73} Yu.~Manin, {\itshape Periods of parabolic forms and $p$-adic Hecke series}, Mat.~Sb., {\bf 21} (1973), 371--393.

\bibitem{Ma16} D.~Ma, {\itshape Period polynomial relations between formal double zeta values of odd weight}, Math.~Ann., {\bf 365} (2016), no.~1, 345--362.





\bibitem{MaTasaka21} D.~Ma, K.~Tasaka, {\itshape Relationships between multiple zeta values of depths 2 and 3 and period polynomials}, Isr.~J.~Math., {\bf 242} (2021), 359--400.

\bibitem{PasolPopa} V.~Pa\c{s}ol, A.A.~Popa, {\itshape Modular forms and period polynomials}, Proc.~Lond.~Math.~Soc., {\bf 107}(3) (2013), 713--743.

\bibitem{Popa} A.A.~Popa, {\itshape Rational decomposition of modular forms}, Ramanujan J., {\bf 26} (2011), 419--435.




\bibitem{Racinet02} G.~Racinet, {\itshape Doubles m\'{e}langes des polylogarithmes multiples aux racines de l\'{u}nit\'{e}}, Publ. Math. IHES 95 (2002), 185--231

\bibitem{Serre} J.-P.~Serre, {\itshape A Course in Arithmetic}, Graduate Texts in Mathematics, {\bf 7}, Springer-Verlag, New York, 1996.



\bibitem{Tasaka16} K.~Tasaka, {\itshape On linear relations among totally odd multiple zeta values related to period polynomials},  Kyushu J.~Math., {\bf 70}(1) (2016), 1--28.




\bibitem{Tasaka20} K.~Tasaka, {\itshape Hecke eigenform and double Eisenstein series}, Proc.~Amer.~Math.~Soc., {\bf 148}(1) (2020), 53--58.

\bibitem{Tasaka22} K.~Tasaka, {\itshape Note on totally odd multiple zeta values}, Math.~J.~Okayama Univ., {\bf 64} (2022), 63--73.


\bibitem{Yamamoto} S.~Yamamoto, {\itshape Interpolation of multiple zeta and zeta-star values}, J.~Algebra, {\bf 385} (2013), 102--114.

\bibitem{YuanZhao16} H.~Yuan, J.~Zhao, {\itshape Bachmann-K\"uhn's brackets and multiple zeta values at level $N$}, Manuscripta Math., {\bf 150} (2016), 177-210.

\bibitem{Zagier91} D.~Zagier, {\itshape Periods of modular forms and Jacobi theta functions}, Invent.~Math., {\bf 104} (1991), 449--465.

\bibitem{Zagier94} D.~Zagier, {\itshape Values of zeta functions and their applications}, First European Congress of Mathematics, Vol. II (Paris, 1992), Progr.~Math., {\bf 120}, Birkh\"{a}user, Basel (1994), 497--512. 

\bibitem{Zagier93} D.~Zagier, {\itshape Periods of modular forms, traces of Hecke operators, and multiple zeta values}, RIMS Kokyuroku, {\bf 843} (1993), 162--170.



\bibitem{Zagier12} D.~Zagier, {\itshape Evaluation of the multiple zeta values $\zeta(2,\ldots,2,3,2,\ldots,2)$}, Ann.~of Math., {\bf 175} (2012), no.~2, 977--1000.



\end{thebibliography}
\end{document}